\documentclass[11pt,reqno,oneside]{amsart}
\UseRawInputEncoding           
          
\usepackage{graphicx, subfigure,}
\usepackage{amssymb}
\usepackage{epstopdf}
\usepackage{amssymb}
\usepackage[a4paper, total={6in, 9.1in}]{geometry}

\usepackage{bm}
\usepackage{amsmath,amsfonts,amsthm,mathrsfs,amssymb,cite}
\usepackage[usenames]{color}

\usepackage{graphics}
\usepackage{framed}

\usepackage{enumerate}
\usepackage{hyperref}
\usepackage{xcolor}
\hypersetup{
  colorlinks,
  linkcolor={blue!80!black},
  urlcolor={blue!80!black},
  citecolor={blue!80!black}
}
\usepackage[capitalize,noabbrev]{cleveref}

\numberwithin{equation}{section}

\newtheorem{theorem}{Theorem}[section]
\newtheorem{lemma}[theorem]{Lemma}
\newtheorem{proposition}[theorem]{Proposition}

\theoremstyle{definition}
\newtheorem{definition}[theorem]{Definition}
\newtheorem{assumption}{Assumption}

\theoremstyle{remark}
\newtheorem{remark}[theorem]{Remark}

\numberwithin{equation}{section}

\newcounter{saveeqn}
\newcommand{\ba}{\begin{array}}
\newcommand{\ea}{\end{array}}
\newcommand{\dvg}{{\rm div\;}}

\newcommand{\bea}{\begin{eqnarray*}}
\newcommand{\eea}{\end{eqnarray*}}
\newcommand{\bean}{\begin{eqnarray}}
\newcommand{\eean}{\end{eqnarray}}

\def\B{\mathcal{B}}
\def\R{\mathbb{R}}
\def\C{\mathbb{C}}

\def\N{\mathbb{N}}
\def\h{\mathfrak{h}}
\def\S{\mathbb{S}}

\def\n{\nabla}

\def\div{\mbox{div}}
\def\dist{\mathrm{dist}}

\def\a{\alpha}
\def\b{\beta}
\def\d{\delta}
\def\e{\varepsilon}
\def\g{\gamma}

\def\t{\theta}

\def\p{\partial}

\def\ds{\displaystyle}

\title[Stability for an inverse scattering problem and more applications]{Stable determination by a single measurement, scattering bound and regularity of transmission eigenfunction}

\author{Hongyu Liu}
\address{Department of Mathematics, City University of Hong Kong, Kowloon, Hong Kong SAR, China}
\email{hongyu.liuip@gmail.com, hongyliu@cityu.edu.hk}

\author{Chun-Hsiang Tsou}
\address{Department of Mathematics, City University of Hong Kong, Kowloon, Hong Kong SAR, China}
\email{chun.hsiang.tsou@gmail.com}

\date{} % Activate to display a given date or no date (if empty),
\begin{document}

\begin{abstract}

In this paper, we study an inverse problem of determining the cross section of an infinitely long cylindrical-like material structure from the transverse electromagnetic scattering measurement. We establish a sharp logarithmic stability result in determining a polygonal scatterer by a single far-field measurement. The argument in establishing the stability result is localised around a corner and can be as well used to produce two highly intriguing implications for invisibility and transmission resonance in the wave scattering theory. In fact, we show that if a generic medium scatterer possesses an admissible corner on its support, then there exists a positive lower bound of the $L^2$-norm of the associated far-field pattern. For the transmission resonance, we discover a quantitative connection between the regularity of the transmission eigenfunction at a corner and its analytic or Fourier extension.

\medskip



\end{abstract}

\maketitle

\section{Introduction}
\subsection{Mathematical setup and physical background}

Initially focusing on the mathematics but not the physics, we introduce the scattering system for our study. Let $k\in\mathbb{R}_+$ be a wavenumber and $u^i$ be an entire solution to the Helmholtz equation,
\begin{equation}\label{eq:helmholtz_inc}
\Delta u^i + k^2 u^i=0\quad\mbox{in}\ \ \mathbb{R}^2. 
\end{equation}
Let $\e, \mu\in L^\infty(\mathbb{R}^2)$ be real-valued functions such that $\mbox{supp}(\e-1), \mbox{supp}(\mu-1)\subset D$ and $\e, \mu\geq c_0\in\mathbb{R}_+$, where $D$ is a bounded Lipschitz domain such that $\mathbb{R}^2\backslash\overline{D}$ is connected. Consider the following scattering system for $u=u^i+u^s\in H_{loc}^1(\mathbb{R}^2)$: 
\begin{equation}\label{eq:aa1}
\begin{cases}
\div (\e^{-1} \n u)+k^2 \mu u=0 & \mbox{in }\R^2,\medskip\\
r^{1/2}(\p_r-\mathrm{i}k)(u-u^i) \rightarrow 0 & \mbox{while } r:=|x|\rightarrow \infty,
\end{cases}
\end{equation}
where $\mathrm{i}:=\sqrt{-1}$. 

In the physical context, \eqref{eq:aa1} describes the transverse magnetic scattering due to the impingement of an incident field $u^i$ on an infinitely long cylindrical material structure whose cross section is $D$; see e.g. \cite{LL1} for more detailed discussion. Here, $\e$ and $\mu$ are respectively the electric permittivity and magnetic permeability, which characterize the medium parameters. The medium outside the material structure is uniformly homogeneous. $u^s$ signifies the perturbation of the incident field due to the presence of the inhomogeneous medium and is referred to as the scattered field. The last limit in \eqref{eq:aa1} is called the Sommerfeld radiation condition which characterizes the outgoing nature of the scattered field. The well-posedness of the scattering system \eqref{eq:aa1} is known (cf. \cite{LSSZ,McL}), and moreover the following asymptotic expansion holds (cf. \cite{ColtonKress}):
\begin{equation}\label{eq:asympt_far}
u(x)=u^i(x)+\frac{e^{\mathrm{i}k|x|}}{|x|^{1/2}}u_\infty(\hat{x};u^i)+\mathcal{O}(|x|^{-1})\ \ \mbox{as}\ \ |x|\rightarrow+\infty,
\end{equation}
which holds uniformly in the angular variable $\hat x:=x/|x|\in\S^1$. The function $u_\infty : \S^1\rightarrow \C$ is called the \textit{far-field pattern} to the scattering problem.

It is noted that the transverse electric scattering can be described similarly with the first equation in \eqref{eq:aa1} replaced by $\dvg (\mu^{-}\nabla u)+k^2\e u=0$. In order to unify our discussion, we introduce the following scattering system 
\begin{equation}\label{eq:helmholtz_total}
\begin{cases}
\div (\sigma \n u)+k^2 q u=0 & \mbox{in }\R^2,\medskip\\
r^{1/2}(\p_r-\mathrm{i}k)(u-u^i) \rightarrow 0 & \mbox{while } r\rightarrow \infty, 
\end{cases}
\end{equation}
where $(\sigma, q)$ may be either $(\e^{-1},\mu)$ or $(\mu^{-1}, \e)$. The inverse scattering problem that we are concerned with associated with the scattering system \eqref{eq:helmholtz_total} in this article is to recover $D$, independent of $\sigma$ and $q$, by knowledge of $u^\infty$ corresponding a single incident field $u^i$, namely
\begin{equation}\label{eq:ip1}
u^\infty\ \mbox{corresponding to a fixed }u^i \longrightarrow D, \ \mbox{independent of } \sigma \mbox{ and }q. 
\end{equation}

\subsection{Summary of main results and connection to existing studies}

The inverse problem \eqref{eq:ip1} is a fundamental one in the inverse scattering theory for electromagnetic waves, and there are rich results in the literature (cf. \cite{ColtonKress,LZr,30Calderon}). However, the determination by a single far-field measurement in the general case still remains to be an open problem; see \cite{Liu} and the references cited therein for more relevant discussions. For the convenience of the readers to have a global picture of our result, we briefly summarize the major discovery in this article in what follows. 

Suppose that $\sigma$ is of the form:
\begin{equation}\label{eq:sigma_form}
\sigma:=1+(\gamma-1)\chi_D 
\end{equation} 
where $\gamma\neq 1$ is a positive constant and $q\in L^\infty$ with $\mbox{supp}(q-1)\subset D$. We write $(D; \g, q)$ to signify such a medium scatterer. The main stability result established in this paper for the inverse problem \eqref{eq:ip1} can be stated as follows. 
\begin{theorem}\label{main_th_simp}
Let $(D; \g, q)$ and $(D'; \g', q')$ be two medium scatterers as described above, where $D$ and $D'$ are assumed to be convex polygons in $\R^2$. Let $u$ and $u'$ the corresponding solutions to the scattering system \eqref{eq:helmholtz_total} associated with $(D; \g, q)$ and $(D'; \lambda', q')$ respectively. Suppose that the solutions $u$ and $u'$ as well as the incident field $u^i$ satisfy the admissible assumptions, which will be detailed in Theorem \ref{main-theorem}. Denote by $u_\infty$ and by $u'_\infty$ the far-filed patterns respectively, and suppose that 
\begin{equation}\label{eq:aa2}
\| u_\infty - u'_\infty \|_{L^2(\S^1)} \leq \e,
\end{equation}
where $\e\in\mathbb{R}_+$ and $\e\ll 1$. 
Then there exist constants $C,\b>0$ which depend only on the {\it a-priori} parameters such that when $\e>0$ is sufficiently small,
\begin{equation}\label{eq:aa3}
d_\mathcal{H}(D, D') \leq C (\ln|\ln\e|)^{-\b}.
\end{equation}
Here $d_\mathcal{H}$, defined later by \eqref{eq:def_haus}, signifies the Hausdorff distance between two domains.
\end{theorem}

It is known that stability estimates of logarithmic type are generically optimal for inverse shape problems (cf. \cite{LMRX,LRX,Linew,Blasten2020}). Hence, it is unobjectionable to claim that our stability result in Theorem~\ref{main_th_simp} is sharp. As mentioned earlier, the determination by a single far-field measurement is a challenging problem in the literature. Our study is closely related to two recent works \cite{cakoni2019,Blasten2020}. In \cite{Blasten2020}, a similar stability estimate was established, but for the case $\sigma\equiv 1$. That is, the determination of the inhomogeneous support of the coefficient in the zeroth-order term of the elliptic PDE in \eqref{eq:helmholtz_total}. In \cite{cakoni2019}, the general case with an inhomogeneous coefficient appearing in the leading-order term was considered, but only a qualitative uniqueness result was established. Hence, our result sharply quantify the uniqueness result in \cite{cakoni2019}. We would like to point out that in \cite{Blasten2020} and \cite{cakoni2019}, both two and three dimensions are considered, whereas in the current article we only consider the two-dimensional case. Moreover, in \cite{cakoni2019}, the coefficient $\gamma$ can be variable but fulfilling certain conditions, whereas in our study, $\gamma$ is assumed to be a constant. There are technical reasons for us doing so as discussed in what follows. 

The mathematical strategy developed in this article is partly inspired from the method used in \cite{Blasten2020}. It begins with an integral identity, which can be established by applying Green's formula on a neighborhood of the scatterer. The second step is to construct a suitable family of functions, which are referred to as the Complex-Geometrical-Optics (CGO) solutions. The third step is to estimate the two sides of the integral identity in order to give an estimation of the geometrical parameters in terms of the estimation of the solutions near the scatterer. We shall use the quantitative estimation of the unique continuation property \cite{Alessandrini} to give an estimation of solutions near the scatterer in terms of the far-field patterns. We finally shall choose delicately the right CGO solution to derive our main stability result. It is particularly noted that there is a challenging point comes from the discontinuity of the leading-order coefficient (cf. \eqref{eq:sigma_form}). To overcome this difficulty, we shall make use of the corner singularity decomposition theory (see Section \ref{Assumptions} for more details). It is noted that the corner singularity in three dimensions can be much more complicated, which include both the edge and the vertex singularities. This makes the singular behaviours of the solutions to elliptic PDEs radically more complicated. This is also the main reason that we are confined within the two dimensions in the current study. It is also the technical reason that we assume that $\gamma$ in \eqref{eq:sigma_form} is constant, since otherwise the induced singularity in the solution due to the singularity of the coefficient coupled with the geometric singularity can be highly complicated. That is, as long as the singularity decomposition theory is well understood, our study can be extended to more general scenarios in high dimensions with variable coefficients. Nevertheless, as can be seen from our subsequent analysis that in principle, we only make use of the fact that $\gamma$ is a constant in a neighbourhood of any corner on $\partial D$ and it can be variable in the rest part of $D$. However, in order to ease the exposition, we assume that $\gamma$ is constant in the whole domain. Here, we would also like to mention in passing some related studies for the so-called Calderon's inverse problem \cite{moi5,moi6} and inverse obstacle problems \cite{CDLZ1,CDLZ2,DLW1,DLW2,DLZZ,LLSS,LLW,Liu2,LYZ,LZ1}.

In addition to its significance in the inverse scattering theory, the stability result can yield interesting implications to the invisibility which is a topic that has received considerable attentions recently. It is shown in \cite{BPS} that if $\mbox{supp}(q)$ (assuming $\sigma\equiv 1$) possesses a corner, then it scatters every incident field nontrivially, namely the corresponding far-field pattern cannot identically vanish. This result is further quantified in \cite{Blasten2020} by showing the corner scatters stably in the sense that the energy of the far-field pattern possesses a positive lower bound. As an immediate consequence of Theorem~\ref{main_th_simp}, one can show that a convex polygonal material structure $(D; \sigma, q)$ scatters stably. The most interesting part is that we can show such a stable scattering result for a scatterer with a generic shape as long as it possesses a corner since our argument in establishing Theorem~\ref{main_th_simp} can be localized around the corner. Another highly intriguing implication is about the analytic or Fourier extension of the transmission eigenfunction. We make a novel discovery on a quantitative relation between the regularity of the transmission eigenfunctions at a corner and their analytic or Fourier extensions.  We shall present more details in Section~5. These intriguing topics have received considerable attentions recently in the literature, and we refer to \cite{B,BLLW,BL1,BL3,curvature2018,BLX2020,CDL,CDHLW,CLW,DDL,DJLZ,DLWW,DCL,DLW,DLS,LX,SPV,SS} as well as a survey paper \cite{Liu} for more existing developments in different physical contexts. 

The rest of this paper is organized as follows. In Section \ref{Assumptions}, we first discuss the corner singularity decomposition, which is a critical ingredient in our study. Then we introduce some admissibility assumptions on the geometric setup as well as the incident fields. In Section \ref{Propagation}, we establish a quantitative estimation of the difference between two wave functions near the scatterer knowing that their far-field patterns are close enough. In Section \ref{Proof_M}, we focus our analysis on a neighborhood of a vertex on the scatterer. We establish in this region the fundamental integral identity. Then by a highly delicate analysis, we derive our stability estimate. In Section \ref{sec:scatter_stable}, we study the stable scattering due to a corner as well as the transmission resonance, which are quantitative byproducts of our stability study.

\section{Geometrical setup and statement of the main results}\label{Assumptions}

\subsection{Corner singularity decomposition}
One of the key ingredients to derive the stability theorem is the local decomposition of solutions to transmission problems in a neighborhood of each polygonal vertex. This is known from the pioneering works of Kondrat'ev \cite{Kon67}, Grisvard \cite{Grisvard} and from Kozlov, Maz'ya, Rossemann \cite{Kozlov} that the solution to a boundary value problem $\Delta u=f$ in a domain with a polygonal corner can be written as the sum of a regular part and a singular part. The singular part is composed of a discrete family of special functions which are explicit in terms of the underlying geometrical parameters. This result is extended to transmission problems, and we refer to Kellogg \cite{Kellogg1}, Dauge and Nicaise \cite{dauge1989oblique,nicaise1990polygonal}. For the boundary value problem associated with the Helmholtz equation system, the same decomposition as the Laplacian case holds \cite{Chaumont-Frelet2018}. The transmission problems for the Helmholtz systems are also studied by Costabel \cite{Costabel1985} in a framework of the layer potential techniques. We summarize those results into the following theorem:
%%%%%%%%%%%%%%%%%%%%%%%%%%%%%%%%%%%%%%%%%%%%%%%%%%%%%%%%%%%%%%%

\begin{theorem}[Local decomposition of solutions to transmission problems]\label{theo_decom_sol}
Let $u\in H^1_{loc}(\R^2)$ be the solution to \eqref{eq:helmholtz_total} with $\sigma$ satisfying \eqref{eq:sigma_form}, in which $D$ represents a convex polygon. We assume also that $u$ satisfies the asymptotic behavior \eqref{eq:asympt_far} at infinity with a nonzero incident field $u^i \in H^2_{loc}(\R^2)$. We denote by $\mathcal{S}_D$ the set of vertices of $D$. Here the variables $r$, $\theta$ are related to the polar coordinates in the neighborhood of each vertex. Then the following decomposition holds,
\begin{equation}\label{eq:decom_sol}
u=u_{reg}+\sum_{x_i\in \mathcal{S}_D} K_i r^{\eta_i}\phi_i(\theta)\zeta_i.
\end{equation}
We have also the following properties:
\begin{enumerate}
\item Let $B_R$ be the central ball with a radius $R>0$ such that $D \Subset B_R$. $u_{reg}|_D \in H^2(D)$, $u_{reg}|_{B_R\setminus\overline{D}}\in H^2(B_R\setminus\overline{D})$, and $u_{reg}$ is continuous across $\p D$.
\item There holds
\begin{equation}\label{eq:estime_u_reg}
\| u_{reg} \|_{H^2(D)}\leq C \| u \|_{H^1(B_{R})},
\end{equation}
where the constant $C$ is independent of the shape of $D$. The same regularity estimate holds for $u_{reg}$ in $B_R\setminus \overline{D}$.
\item There exists explicit formulas to calculate the coefficients $K_i$.
\item The exponent $\eta_i\in (0,1)$ satisfies the following equation
\begin{equation}\label{eq:eta_gamma}
\left(\frac{\sin \eta_i(\pi-a_i)}{\sin \eta_i \pi}\right)^2=\left(\frac{\gamma+1}{\gamma-1}\right)^2,
\end{equation}
where $a_i$ stands for the opening angle at each vertex $x_i$.
\item The functions $\phi_i$ admit the following form: $\phi_i(\t)=\cos(\eta_i\t+\Phi_{i,\pm})$, where the phase shits $\Phi_{i,\pm}$ are constants determined by the transmission conditions.
\item $\zeta_i$ is a smooth cut-off function such that $\zeta_i(r)=1$ if $r\leq\varrho_i$ and $\zeta_i(r)=0$ if $r\geq 2\varrho_i$. The radius $\varrho_i$ is chosen such that the disks $(B_{2\varrho_i}(x_i))_{i\in\mathcal{S}_D}$ do not intersect each other.
\end{enumerate}
\end{theorem}
%%%%%%%%%%%%%%%%%%%%%%%%%%%%%%%%%%%%%%%%%%%

\begin{proof}
The solution to equation \eqref{eq:helmholtz_total} in the ball $B_R$ can be rewritten as the solution to the following system,
\begin{equation}\label{eq:aux}
\begin{cases}
\div (\sigma \n u)=\hat{f} &\mbox{in } B_R,\medskip \\
u=\hat{g} &\mbox{on } \p B_R,
\end{cases}
\end{equation}
where $\hat{f}=-k^2 q u \in L^2(B_R)$, $\hat{g}=u|_{\p B_R} \in H^{1/2}(B_R)$.

Then it follows from Theorem 1.6 in \cite{nicaise1990polygonal} that the solution to the transmission problem \eqref{eq:aux} admits the decomposition \eqref{eq:decom_sol}. The equation \eqref{eq:eta_gamma} to determine the exponents $\eta_i$ can be obtained by following similar calculations as those in \cite{dauge2011texier}. From the equation (4.3bis) in \cite{nicaise1990polygonal}, we also have an estimation of the regular part $u_{reg}$,
\begin{equation}
\| u_{reg}\|_{H^2(D)} \leq C (\| \hat{f} \|_{L^2(B_R)}+\| \hat{g} \|_{H^{1/2}(\p B_R)}),
\end{equation}
with a constant $C$ independent of the shape of $D$. It follows then from the trace operator on $\p B_R$ that the estimate \eqref{eq:estime_u_reg} holds. Moreover, the coefficients $K_i$ can be explicitly calculated by the formula
 (1.24) in \cite{nicaise1990polygonal}
\end{proof}
%%%%%%%%%%%%%%%%%%%%%%%%%%%%%%%%%%%%%%%%%%%%%%%%%%%%%
\subsection{Admissibility assumptions}
Before announcing our main stability result, we introduce some admissibility conditions to clarify the framework of our study.

\begin{definition}\label{def:class}
Let $D$ be a convex polygon in $\R^2$, $\gamma>0$ and $q \in L^\infty(\R^2)$. We say that $(D;\gamma, q)$ belongs to the admissible class $\mathcal{D}$ if the following conditions are fulfilled:
\begin{enumerate}
\item $D$ is a convex polygon and $\gamma \in\mathbb{R}_+$ with $\gamma \neq 1$ and $\gamma \in (\gamma_m, \gamma_M)$, where $\gamma_m$ and $\gamma_M$ are two positive constants;
\item $D \Subset B_R$, where $B_R$ stands for the central ball at origin with a radius $R>0$;
\item There exist $0<a_m<a_M<\pi$ such that the opening of the angle at each vertex of $D$ is in $(a_m,a_M)$;
\item The length of each edge of $D$ is at least $l>0$;
\item The support of $q-1$ is contained in the polygon $D$, which means $q \equiv 1$ in $\R^2 \setminus \overline{D}$;
\item $\| q \|_{L^\infty(\R^2)} \leq \mathcal{Q}$ where $\mathcal{Q}>0$ is a constant.
\end{enumerate}
\end{definition}
%%%%%%%%%%%%%%%%%%%%%%%%%%%%%%%%

%\begin{definition}\label{def:wellposed}
%Let $\sigma$ satisfy \eqref{eq:sigma_form} with $(D,\gamma)\in \mathcal{D}$. A potential $q\in L^\infty(\R^2)$ is said to give a {\it well-posed scattering problem} if there exist $S>0$ such that there exists a unique solution $u\in H^1_{loc}(\R^2)$ to the Helmholtz equation \eqref{eq:helmholtz_total} for all incident plane wave $u^i=\exp(ik\omega\cdot x)$. Moreover, $u$ verifies
%\begin{equation}\label{norm_H1_u}
%\| u \|_{H^1(B_{2R})}\leq S.
%\end{equation}
%\end{definition}
%%%%%%%%%%%%%%%%%%%%%%%%%%%%%%%%%%
%We shall take only a limited class of the potentials $q$ into consideration in the rest of this paper, we summarize the way we choose our admissible potentials by the following assumptions.

%\begin{assumption}\label{assum_support}
%Let $(D,\gamma)\in \mathcal{D}$ and $q\in L^\infty(\R^2)$ gives a well-posed scattering problem, we assume that the support of $q-1$ is a sub-domain of $D$, namely, $q \equiv 1$ in $\R^2 \setminus \overline{D}$.
%\end{assumption}

\begin{remark}
In fact, the inhomogeneity on both $\sigma$ and $q$ could generate the scattering waves. It is shown in \cite{Blasten2020} that the support of $q-1$ can be stably determined from a single far-field pattern under the assumption $\sigma \equiv 1$. In principle, it is possible to stably recover both the supports of $\sigma$ and $q-1$ by using in parallel the method introduced in this paper and the one in \cite{Blasten2020}. However, if we do so, the technical details will become more complicated and distract the meaning of introducing our new method. In order to have a focusing theme of our study, we are only concerned with the recovery of the support of $\sigma$ in this paper. This is the main reason for us introducing item $(5)$ in Definition \ref{def:class}.
\end{remark}

On the other hand, we shall impose a genetic condition on the incident field $u^i$ such that the singular behaviors at each corner are guaranteed.

\begin{assumption}\label{def_admis_inc}
Let $u^i \in H^2_{loc}(\R^2)$ be an entire solution to \eqref{eq:helmholtz_inc}. We denote by $S>0$ the {\it amplitude} of the incident wave $u^i$, which is defined by 
\begin{equation}\label{eq:h2apriori}
\| u^i \|_{H^2(B_{2R})} \leq S,
\end{equation}
with $R>0$ introduced in Definition \ref{def:class}. We assume in the rest of this paper that for all $(D; \g,q) \in \mathcal{D}$, the corresponding solution $u$ to the scattering problem \eqref{eq:helmholtz_total}--\eqref{eq:asympt_far} admits a nondegenerate singularity coefficient $K_i>0$ in the decomposition \eqref{eq:decom_sol} at each vertex $x_i$ of $D$.
\end{assumption}

%\begin{definition}
%Let $K_0, S \in \mathbb{R}_+$, $u^i \in H^2_{loc}(\R^2)$ an entire solution to \eqref{eq:helmholtz_inc}. We assume firstly that
%\begin{equation}\label{eq:h2apriori}
%\| u^i \|_{H^2(B_{2R})} \leq S,
%\end{equation}
%with $R>0$ introduced in Definition \ref{def:class}.
%
%We say that $u^i$ is a $K_0$-admissible incident field with respect to the class of polygonal inclusions $\mathcal{D}$ if for any $(D; \g,q) \in \mathcal{D}$, the corresponding solution $u$ to the scattering problem \eqref{eq:helmholtz_total}--\eqref{eq:asympt_far} admits a coefficient $K_i$ in the decomposition \eqref{eq:decom_sol} satisfying $1/K_0\leq K\leq K_0$ for each vertex $x_i$ of $D$.
%\end{definition}

\begin{remark}
It is pointed out that \eqref{eq:h2apriori} is a generic condition, which can be easily fulfilled, say e.g. by the plane wave of the form $\exp\{\mathrm{i}kx\cdot d\}$, $d\in\S^1$. 

As we mentioned previously in Theorem \ref{theo_decom_sol}, the coefficients $K_i$ in \eqref{eq:decom_sol} can be determined by explicit formulas. The formulas can be found in \cite{dauge1989oblique,nicaise1990polygonal}. In fact, each $K_i$ depends linearly on the incident field $u^i$ for fixed polygon $D$ and vertex $x_i$. Assumption \ref{def_admis_inc} is in fact a generic condition except for the case when $K_i=0$, which means the solution to \eqref{eq:helmholtz_total} has $H^2$-regularity in a neighborhood of $x_i$, corresponding to a very limited class of scenarios from the practical point of view.
\end{remark}

In what follows, for $D, D'\in\mathcal{D}$, we define
\begin{equation}\label{eq:def_haus}
d_\mathcal{H}(D, D')=\max \big(\sup_{x\in D}\mathrm{dist}(x,D'),\sup_{x'\in D'}\mathrm{dist}(x',D)\big),
\end{equation}
to be Hausdorff distance between $D$ and $D'$. In order to simplify the expressions, we shall call the following parameters {$(k, R, \gamma_m,\gamma_M,a_m,a_M,l , \mathcal{Q})$} as the {\it a-priori} parameters. With those admissible assumptions and {\it a-priori} parameters, we are in a position to complete the statement of Theorem \ref{main_th_simp}.

%%%%%%%%%%%%%%%%%%%%%%%%%%%%%%%%%%%%%%%%%%%%%%%
\begin{theorem}\label{main-theorem}
Let $k \in\mathbb{R}_+$ and $u^i \in H^2_{loc}(\R^2)$ be a nontrivial solution to the Helmholtz equation \eqref{eq:helmholtz_inc}. Let $(D;\gamma,q)$ and $(D';\gamma', q')$ belong in the admissible class $\mathcal{D}$ in Definition \ref{def:class}. We consider the solutions $u,u'\in H^1_{loc}(\R^2)$ to the Helmholtz equation \eqref{eq:helmholtz_total} where the corresponding functions $\sigma, \sigma'$ are defined by \eqref{eq:sigma_form}. {Suppose that $u^i$ satisfies Assumption \ref{def_admis_inc}. We define $\ds K_m:=\min_{x_i\in\mathcal{S}_D \cup \mathcal{S}_{D'}} |K_i|>0$ as the smallest singularity coefficient among the vertices of $D$ and $D'$.}
%Let $K_0>0$. We assume that $u^i$ is a $K_0$-admissible incident field as in Definition \ref{def_admis_inc}.

Assume that there holds
\begin{equation}\label{eq:small}
\| u_\infty - u'_\infty \|_{L^2(\S^1)} \leq \e.
\end{equation}
Then there exist constants $C,\b,\widetilde{\b}>0$ which depend only on the {\it a-priori} parameters such that when $\e\in\mathbb{R}_+$ is sufficiently small, one has
\begin{equation}\label{eq:stability}
d_\mathcal{H}(D, D') \leq C \left(1+\frac{S}{K_m}\right)^{\widetilde{\b}}\left(\ln\ln \frac{S}{\e}\right)^{-\b}.
\end{equation}
\end{theorem}
%%%%%%%%%%%%%%%%%%%%%%%%%%%%%%%%%%%%%%%%%%%%%%
\section{Propagation of smallness}\label{Propagation}

The proof of Theorem~\ref{main-theorem} involves several technical ingredients. The objective of this section is to estimate the difference $u-u'$ of the solutions to the scattering problems near the polygonal scatterer $D$. The first step in our method is to estimate $u-u'$ in terms of the difference on their far-field patterns $u_\infty-u'_\infty$ in a near-field domain. The second step is to estimate $u-u'$ and its derivative $\n u-\n u'$ in a neighborhood of the polygonal scatterer $D$ from the near-field estimation. Combining those results, we derive the estimations of $u-u'$ and $\n u-\n u'$ in terms of $u_\infty-u'_\infty$.

\begin{proposition}\label{propagation1}
Let $w^s\in H^2_{loc}(\R^2)$ be a solution to \eqref{eq:helmholtz_inc} in $\R^2 \setminus B_{R}$ and satisfy the Sommerfeld radiation condition \eqref{eq:asympt_far} at infinity. We denote by $w^s_\infty$ its far-field pattern and $\e=\| w^s_\infty \|_{L^2(\S^1)}$.

We assume the {\it a-priori} bound $\| w^s \|_{L^2(B_{2R}\setminus B_{R})}\leq \mathcal{S}$ with $\mathcal{S} \geq 0$ depending on the {\it a-priori} parameters. Let $\mathcal{A}$ be a domain such that $\mathcal{A} \subset B_{2R}\setminus \overline{B_{R}}$. Then, for any smoothness index $p\in \N$ there exists constants $c,C >0$ depending only on $k,p,R,\mathcal{A}$ such that
\begin{equation}\label{eq:far_to_near}
\| w^s \|_{H^p(\mathcal{A})} \leq C \max(\e, \mathcal{S}e^{-c\sqrt{\ln(\mathcal{S}/\e)}}).
\end{equation}
\end{proposition}
\begin{proof}
See Proposition 5.2 and Corollary 5.3 in \cite{Blasten2020}.
\end{proof}

We next present the propagation of smallness from a near field domain up to the boundary of the polygonal scatterer. The method we use here is mainly inspired by the one in \cite{Blasten2020}, and we shall adopt similar notations therein.

%The method we use here, which consists of a sequence of three-sphere inequalities, was well developed in the context of Cauchy problems \cite{Alessandrini}. The main ingredients of the following proposition are well detailed in Lemma 5.4 - Proposition 5.8 in \cite{Blasten2020}, we adapt as well as the notations therein.
%\begin{lemma}
%There exists positive constants $R_m,C,c_1$ such that $0<c_1<1$, which depend only on $k$. For $x\in \R^2$, $0<4r<R_m$, if $w\in H^2_{loc}(\R^2)$ satisfying \eqref{eq:helmholtz_inc} in $B_{4r}:=B(x,4r)$, then,
%\begin{equation}\label{eq:3_sphere}
%\| w \|_{L^\infty(B_{2r})} \leq C(2+\sqrt{2})^{3/2} \| w \|_{L^\infty(B_{4r})}^{1-\beta} \| w \|_{L^\infty(B_r)}^\beta,
%\end{equation}
%where $\beta$ satisfies
%\[\frac{c_1}{4}\leq \beta \leq 1-\frac{3c_1}{4}.\]
%\end{lemma}
%\begin{proof}
%See Lemma 5.4 in \cite{Blasten2020} and Lemma 3.5 in \cite{rondi2008}.
%\end{proof}
\begin{proposition}\label{propagation2}
Let $Q\Subset B_R\subset \R^2$ be a convex polygon, $x_c \in \p Q$ be a vertex of $Q$, $P\in \N$, $0<\a<1$ and $w\in H^1(B_{2R}\setminus \overline{Q})$. We assume that $w$ is a solution to \eqref{eq:helmholtz_inc} in $B_{2R}\setminus \overline{Q}$ and the function $\widetilde{w_P}: x \mapsto |x-x_c|^P w(x)$ belongs to the class $\mathcal{C}^\a$ in $B_{2R}\setminus \overline{Q}$ with a norm at most $T\geq 1$.\\

We assume furthermore that $|w(x)|\leq \d$ in $B_{7R/4}\setminus \overline{B_{5R/4}}$. If 
\begin{equation}\label{critere_d}
\d\leq \d_m= \left[ \exp \exp \left( \frac{9R|\ln c_2|}{(1-\a)\min(R_m,2R/5)}\right) \right]^{-1},
\end{equation}
then there exists a constant $C_0>0$ depending only on $k$, $\a$, $P$ and $R$ such that,
\begin{equation}\label{eq:ee1}
|w(x)|\leq \frac{C_0T}{\dist(x,\p Q)^P}(\ln|\ln \d|)^{-\a},
\end{equation}
for $x\in B_{3R/2}\setminus \overline{Q}$. Here $R_m$ and $c_2$ are constants, which depend only on $k$ and are given respectively in Lemmas 5.4 and 5.5 in \cite{Blasten2020}. 
\end{proposition}
\begin{proof}
Define
\[r=r(\d)=\frac{9R|\ln c_2|}{4(1-\a)\ln|\ln \d|}>0.\]
It follows from the assumption of the function $\widetilde{w_P}$ that $w \in L^\infty(B_{2R}\setminus B(Q,r))$ and $\| w \|_{L^\infty(B_{2R}\setminus B(Q,r))} \leq r^{-P}T$.
It follows in parallel from the upper bound \eqref{critere_d} of $\d$ that $4r<R_m$ and that $2r<2R/5<R$. We can thus apply Proposition 5.7 in \cite{Blasten2020} where the convex polygon $Q$ is replaced by the convex set $B(Q,r)$ with the parameter $\lambda=1/4$. Then it gives
\begin{equation}\label{eq:passage0}
|w(x')|\leq CT\left(\frac{1}{r}\right)^P\d^{c_2^{9R/(4r)+2}},
\end{equation}
for $x'\in B_{2R}$ and $\mathrm{dist}(x',\p Q)\geq 5r$. The constant $C$ is given in Lemma 5.5 in \cite{Blasten2020}.

We now assume $\dist(x,\p Q)< 5r$, then there exists $y\in \p Q$ such that $|x-y|\leq 5r$. By the convexity of $Q$, there exists $x'\in \R^2 \setminus Q$ such that $\dist(x',\p Q)=|x'-y|=5r$. {The upper bound of $\d$ implies $5r\leq R/2$, and thus $|x'|\leq |x'-y|+|y|\leq 5r+R < 2R$. Then  \eqref{eq:passage0} holds at the point $x'$. We have also at the same time, $|x-x'|\leq |x-y|+|y-x'|\leq 10r$.

By virtue of the H\"older continuity of $\widetilde{w_P}$ and the above observations}, we can deduce as follows
\begin{align}\label{eq:passage1}
|w(x)| & \leq \frac{1}{|x-x_c|^P}\left( \|\widetilde{w_P}\|_{\mathcal{C}^\a(B_{2R \setminus \overline{Q}})} |x-x'|^\a + |x'-x_c|^P |w(x')| \right),\nonumber\\
& \leq \frac{1}{|x-x_c|^P}\left( (10r)^\a T+M_P (|x-x'|^P+|x-x_c|^P) CT\left(\frac{1}{r}\right)^P \d^{c_2^{9R/(4r)+2}} \right),\nonumber\\
&\leq \frac{(10r)^\a T}{|x-x_c|^P}+CM_P T\left(\frac{10^P}{|x-x_c|^P}+\left(\frac{1}{r}\right)^P \right) \d^{c_2^{9R/(4r)+2}},\nonumber\\
& \leq \frac{T}{\dist(x,\p Q)^P}\left((10r)^\a + CM_P(10^P+5^P) \d^{c_2^{9R/(4r)+2}} \right),
\end{align}
where $M_P$ is a positive constant such that $(a+b)^P \leq M_P (a^P+b^P)$ holds for all $a,b\in \R_+$.

From the definition of $r(\d)$, we have
\begin{align}\label{eq:passage2}
r^\a &=(\frac{9R|\ln c_2|}{4(1-\a)})^\a (\ln|\ln \d|)^{-\a}, \nonumber \\
 \frac{9R}{4r} &=-\frac{1-\a}{\ln c_2}\ln|\ln \d|=\log_{c_2}(|\ln \d|^{-(1-\a)}).
\end{align}
From the condition $|\ln \d|>1$, we further obtain
\begin{equation}\label{eq:passage3}
\d^{c_2^{9R/(4r)+2}} = e^{-|\ln \d|c_2^{9R/(4r)+2}}= e^{-c_2^2|\ln \d|^{1-(1-\a)}} \leq \frac{1}{c_2^2|\ln \d|^\a} \leq \frac{1}{c_2^2}(\ln|\ln \d|)^{-\a}. 
\end{equation}
Combining \eqref{eq:passage1},\eqref{eq:passage2},\eqref{eq:passage3} and setting
\begin{equation}\label{valeur_C0}
C_0=\left(\frac{45R|\ln c_2|}{2(1-\a)}\right)^\a+\frac{CM_P(5^P+10^P)}{c_2^2},
\end{equation}
one can readily verify that the estimate \eqref{eq:ee1} holds for $\dist(x,\p Q)< 5r$.

We proceed to treat the case $\dist(x,\p Q) \geq 5r$. By following a similar argument in Proposition 5.7 in \cite{Blasten2020} with the convex polygon $Q$ replaced by the convex set $B(Q,\dist(x,\p Q)-4r)$ and the parameter $\lambda =1/4$, one can derive that 
\begin{equation}\label{eq:passage4}
|w(x)|\leq CT\left( \frac{1}{\dist(x,\p Q)-4r} \right)^P \d^{c_2^{9R/(4r)+2}}.
\end{equation}
The assumption $\dist(x,\p Q) \geq 5r$ implies $\left( \frac{1}{\dist(x,\p Q)-4r} \right)^P \leq \left( \frac{5}{\dist(x,\p Q)} \right)^P$. On the other hand, it is clear that $M_P \geq 1$ and $0<c_2<1$. Thus, the desired estimate \eqref{eq:ee1} follows from \eqref{eq:passage3} and \eqref{eq:passage4} with the constant $C_0$ evaluated by \eqref{valeur_C0}.

The proof is complete.
\end{proof}
%%%%%%%%%%%%%%%%%%%%%%%%%%%%%%%%%%%%%%%%%%%%%%
\begin{proposition}\label{propagation3}
Let $u,u'\in H^1_{loc}(\R^2)$ be the solutions of the scattering problems \eqref{eq:helmholtz_total} under the assumptions of Theorem \ref{main-theorem}. Let $Q$ be the polygonal convex hull of $D$ and $D'$, and $x_c$ be a vertex of $Q$. We assume that $u-u'$ is of class $\mathcal{C}^{\a_0}$ in $B_{2R}\setminus \overline{Q}$ and that the function $x\mapsto |x-x_c| \n (u-u')(x)$ is of class $\mathcal{C}^{\a_1}$ in $B_{2R}\setminus \overline{Q}$ where $0<\a_0,\a_1<1$. The corresponding H\"older norms are denoted respectively by $T_0:=\| u-u' \|_{\mathcal{C}^{\a_0}(B_{2R}\setminus \overline{Q})}$ and $T_1:=\| |x-x_c| \n(u-u') \|_{\mathcal{C}^{\a_1}(B_{2R}\setminus \overline{Q})}$.

Then there exists $\e_m>0$ depending only on $k, \g, R_m, R,\a_0,\a_1, \mathcal{Q}, S$ and $\widetilde{C_j}(k,R,\a_j)>0$, $j=0,1$ such that if $\| u_\infty-u'_\infty\|_{L^2(\S^1)}\leq \e\leq \e_m$, it holds that
\begin{align}
|u(x)-u'(x)|&\leq \widetilde{C_0}T_0 \left(\ln\ln \frac{S}{\e} \right)^{-\a_0},\label{eq:u-u'}\\
|\n (u-u')(x)| &\leq \frac{\widetilde{C_1}T_1}{\dist(x,\p Q)} \left(\ln\ln \frac{S}{\e} \right)^{-\a_1},\label{eq:grad_u-u'} 
\end{align}
for $x\in B_{3R/2}\setminus \overline{Q}$.
\end{proposition}

\begin{proof}
The well-posedness of the forward problem stands that for any solution $v$ to the system \eqref{eq:helmholtz_total}, it holds that 
\begin{equation}\label{eq:well_forward}
\| v \|_{H^1(B_{2R})} \leq C(R, \g, |D|, \| q \|_{L^\infty(\R^2)}) \| u^i \|_{H^2(B_{2R})},
\end{equation}
for some $C(R, \g, |D|, \| q \|_{L^\infty(\R^2)})>0$ independent of the geometrical shape of $D$. Thus, it follows from Definition \ref{def_admis_inc} and \eqref{eq:h2apriori} that
\begin{equation}\label{eq:well_forward2}
\| u-u' \|_{H^1(B_{2R})} \leq \| u \|_{H^1(B_{2R})}+\| u' \|_{H^1(B_{2R})} \leq  C_f S,
\end{equation}
where $C_f>0$ depends only on $R$, $\gamma_M$, $\mathcal{Q}$.

We next apply Proposition \ref{propagation1} on the region $\mathcal{A}=B_{7R/4}\setminus \overline{B_{5R/4}}$ with $w^s=u-u' \in H^2_{loc}(\R^2\setminus B_R)$. At the same time, we use also the Sobolev embedding $H^1 \hookrightarrow L^\infty$ in $\R^2$. We then have, 
\begin{equation}\label{eq:passage5}
|u-u'|_{L^\infty(\mathcal{A})},|\n (u- u')|_{L^\infty(\mathcal{A})}\leq  C(R) \| u-u' \|_{H^2(\mathcal{A})} \leq C\max(\e, C_f S e^{-c\sqrt{\ln(C_f S/\e)}}),
\end{equation}
where $c,C>0$ depend only on $k,R$. We can choose now 
\begin{equation}\label{eq:passage6}
\e_m \leq C_f S e^{-c^2},
\end{equation}
such that \eqref{eq:passage5} holds for the second term in the maximum function if $\e \leq \e_m$.

Defining
\[\d =C C_f S e^{-c\sqrt{\ln(C_f S/\e)}},\]
we can therefore choose $\e_m>0$ depending only on $k,R,\gamma_M,\a,\mathcal{Q},S$ such that \eqref{eq:passage6} and \eqref{critere_d} hold. Then we apply Proposition \ref{propagation2}, and it follows that for $x\in B_{3R/2}\setminus \overline{Q}$,
\begin{align*}
&|u(x)-u'(x)|\leq \widetilde{C_0}T_0 \left(\ln |\ln \d|\right)^{-\a_0},\\
&|\n(u-u')(x)|\leq \frac{\widetilde{C_1}T_1}{\dist(x,\p Q)} \left(\ln |\ln \d|\right)^{-\a_1},
\end{align*}
where the constants $\widetilde{C_0}$ and $\widetilde{C_1}$ are given by \eqref{valeur_C0}, which depend only on $k$, $\a_0$, $\a_1$, and $R$. Our choice of $\d$ implies the following inequalities,
\begin{align*}
    |\ln \d|&=c\sqrt{\ln \frac{C_f S}{\e}}-\ln(CC_f S) =c\sqrt{\ln C_f +\ln \frac{S}{\e}}-\ln(CC_f S)\\
    &\geq \frac{c}{2}\sqrt{\ln \frac{S}{\e}} \geq \left(\ln \frac{S}{\e}\right)^{1/4},
\end{align*}
under the condition that $\e_m$ is small enough. Here, we can update the value of $\e_m$, which depends only on $k,R,\gamma_M,\a,\mathcal{Q},S$. Thus, for $\a=\a_0$ or $\a_1$,
\[\left(\ln|\ln \d|\right)^{-\a}\leq \left(\ln(\ln \frac{S}{\e})^{1/4}\right)^{-\a}=4^\a \left(\ln \ln \frac{S}{\e}\right)^{-\a}.\]
By updating the values of $\widetilde{C_0}$ and $\widetilde{C_1}$, \eqref{eq:u-u'} and \eqref{eq:grad_u-u'} hold.

The proof is complete.
\end{proof}
%%%%%%%%%%%%%%%%%%%%%%%%%%%%%%%%%%%%%%%%%%%%%%
\section{Local analysis and proof of the stability result}\label{Proof_M}
In this section, we analyze locally the behavior of the solutions $u,u'$ near a polygonal corner point in order to link the geometrical parameters and the estimations of $u-u'$. This in turn enables to prove the main stability result. 

\subsection{Microlocal analysis near a vertex}\label{LocalA}
\begin{lemma}\label{lemme_geo}
Let $D,D'\subset \R^2$ be two open bounded convex polygons. Let $Q$ be the convex hull of $D\cup D'$. If $x_c$ is a vertex of $D$ such that $\mathrm{dist}(x_c,D')=\h$, where $\h$ gives the Hausdorff distance,
\begin{equation}\label{eq:haus}
\h=d_\mathcal{H}(D, D'),
\end{equation}
then $x_c$ is a vertex of $Q$. If the angle of $D$ at $x_c$ is $a$, then the angle of $Q$ at $x_c$ is at most $(a+\pi)/2<\pi$.
\end{lemma}
%%%%%%%%%%%%%%%%%%%%%%%%%%%%%%%%%%%%%%%%%%%%%%%%%%%%%%%%%%%%%%

\begin{proof}
See the appendix in \cite{Blasten2020}.
\end{proof}
%%%%%%%%%%%%%%%%%%%%%%%%%%%%%%%%%%%%%%%%%%%%%%%%%%%%%%%%%%%%%%
We proceed to present our microlocal analysis around a corner. Let $D$, $D'\in \mathcal{D}$ be two admissible polygons (cf. Definition ~\ref{def:class}). We assume from now on that $D\neq D'$. Let $x_c\in \overline{D}$ be a vertex in Lemma \ref{lemme_geo}. Then there exists $h\in\mathbb{R}_+$ such that $B(x_c,h)\cap D'=\emptyset$. Let $Q$ be the convex hull of $D\cup D'$ and $\B$ be the open disk $B(x_c,h)$. We denote respectively by $\widetilde{D}$ and $\widetilde{Q}$ the sectors $\B\cap D$ and $\B\cap Q$. Let $b$ signify the opening of the angle of $Q$ at $x_c$. Then one has $a_m \leq b \leq (a_M+\pi)/2<\pi$ by Lemma \ref{lemme_geo}. We choose the polar coordinate system such that $x_c$ is the origin point and $\widetilde{Q}$ coincides with the following sector,
\[\widetilde{Q}=\{(r,\theta)|0<r< h,-b/2 < \theta < b/2\}.\]
We define at the same time the unit vectors $\hat{x}$ and $\hat{y}$ to represent respectively the directions $\t=0$ and $\t=\pi/2$.

We next define the integral contours on which we derive the estimates; see Figure \ref{coin} for a schematic illustration. Let
\begin{eqnarray}
\ds \Gamma^\pm &:=& \p D \cap \B,\nonumber\\
\ds \p S^i_D &:=& \p \B \cap D \subset \p S^i_Q,\nonumber\\
\ds \p S^i_Q &:= & \{(r,\theta)| r=h,-\frac{\pi+b}{4}\leq \theta\leq \frac{\pi+b}{4}\},\nonumber
\end{eqnarray}
and $\p S^e$ be a circular arc passing through the following three points in the polar coordinate: $(h,\frac{\pi+b}{4})$, $(-\frac{1}{\tau},\pi)$, $(h,-\frac{\pi+b}{4})$ for a $\tau>0$. 
%Here, the idea is to construct the contour $\p S^e$ such that $(x-x_c)\cdot \hat{x}\geq -\frac{1}{\tau}$ and $\mathrm{dist}(x,\p Q)\geq \frac{1}{2\tau}$ for all $x \in \p S^e$.
We define moreover by $\widetilde{D}^e$ the region surrounded by the closed contour $\overline{\Gamma^\pm \cup (\p S^i_Q \setminus \p S^i_D) \cup \p S^e}$ (the complementary part of $\widetilde{D}$; cf. Figure \ref{coin}). 
\begin{figure}[!ht]
\includegraphics[scale=2.2]{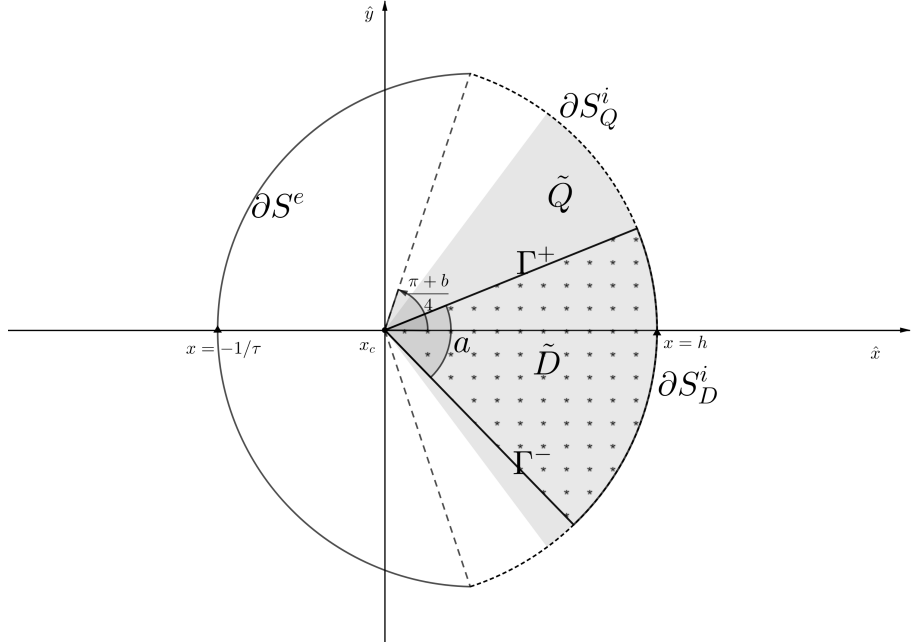}
\caption{The sectors $\widetilde{D}$, $\widetilde{Q}$ and the integral contours.}
\label{coin}
\end{figure}
{
\begin{remark}\label{contour}
There are other ways to construct the integral contours here, especially for $\p S^e$. The idea to construct such contours is to derive some suitable properties we can gain in the next steps. From the above construction, together with direct geometric arguments, we can obtain the following propperties:
\begin{enumerate}
\item for $x \in \widetilde{Q}$, $(x-x_c)\cdot \hat{x} \geq \cos(\frac{\pi+b}{4})|x-x_c|$;
\item for $x\in \widetilde{D^e}\cup \p S^e$, $(x-x_c)\cdot \hat{x} \geq -\frac{1}{\tau}$;
\item if $\tau \geq \tau_0:=C_{\tau_0}h^{-1}=1/[2h\sin(\frac{\pi-b}{4})]$, then $\dist(x,\p Q) \geq \frac{1}{2\tau}$ for $x\in \p S^e$.
\end{enumerate}
These properties and notations will be needed and used in our subsequent analysis.
\end{remark}
}
\begin{proposition}
Let $u,u'$ be the solutions to \eqref{eq:helmholtz_total} in Theorem ~\ref{main-theorem} and $u_0$ be a harmonic function in $\B$. Then it holds that
\begin{multline}\label{IntId}
(1-\gamma)\int_{\Gamma^\pm}u_0\p_\nu u\, ds=\\
\int_{\p S^i_Q \cup \p S^e} (u-u')\p_\nu u_0 - u_0\p_\nu(u-u')\, ds - k^2\int_{\widetilde{D}^e}(u-u')u_0 dx-\frac{k^2 q}{\gamma}\int_{\widetilde{D}}u_0 u\, dx,
\end{multline}
where the integral over $\Gamma^\pm$ is taken the values of $\p_\nu u$ in the interior of $\widetilde{D}$.
\end{proposition}
%%%%%%%%%%%%%%%%%%%%%%%%%%%%%%%%%%%%%%%%%%%%%%%%%%%

\begin{proof}
This proposition follows from the transmission conditions of $u$ and $u'$ respectively on $\partial D$ and $\partial D'$ as well as Green's formula along with straightforward calculations.
\end{proof} 

We next introduce a special type of harmonic functions which are the so-called complex geometric optics (CGO) solutions. In this paper, we define the CGO solution as follows. Let $\tau>0$, we choose $\rho=\rho(\tau):=\tau(-\hat{x}+\mathrm{i}\hat{y})\in \C^2$. For all $x\in \R^2$,
\begin{equation}\label{CGO}
u_0(x)=e^{\rho \cdot (x-x_c)}.
\end{equation}
It is easy to check that $\rho \cdot \rho=0$ and thus $u_0$ is harmonic in $\R^2$.

\begin{proposition}\label{Propo_Upper}
Let $u,u'$ be the solutions to \eqref{eq:helmholtz_total} and satisfy the assumptions in Theorem ~\ref{main-theorem}. Let $\tau>0$ and $u_0$ be a CGO solution defined by (\ref{CGO}). We denote by $\Gamma^\pm_\infty$ the two rays originating from the origin and extending the segments $\Gamma^\pm$ to infinity (cf. Figure \ref{coin}). Then it holds that
\begin{align}\label{IntId2}
\int_{\Gamma^\pm_\infty}  u_0 \p_\nu u_{sing}\, ds &=\ds\int_{\Gamma^\pm_\infty \setminus \Gamma^\pm}u_0 \p_\nu u_{sing}\, ds+\int_{\p S^i_D}u_0\p_\nu u_{reg}\, ds \nonumber \\
 &-\int_{\widetilde{D}}\n u_0\n u_{reg}\, dx + \frac{1}{1-\gamma}\int_{\p S^i_Q \cup \p S^e} (u-u')\p_\nu u_0-u_0\p_\nu(u-u')\, ds \nonumber \\
 & - \frac{k^2}{1-\gamma}\int_{\widetilde{D}^e}(u-u')u_0\, dx - \frac{k^2 q}{1-\gamma}\int_{\widetilde{D}}u_0 u\, dx,
\end{align}
where $u_{sing}:=K r^\eta \phi(\t)$ and $u_{reg}$ represent the singular decomposition \eqref{eq:decom_sol} of $u$ near $x_c$. Here, we assume the characteristics in \eqref{eq:decom_sol} of $u_{sing}$ and $\zeta$ satisfy $0<\eta_m\leq\eta \leq \eta_M<1$ and $\varrho \geq h$. Then the following estimate holds,
\begin{align}\label{UpperBonds}
& C|\int_{\Gamma^\pm_\infty}u_0 \p_\nu u_{sing} ds |\leq |K|\tau^{-\eta}e^{-\a'\tau h/2}+\tau^{-1}\| u_{reg}\|_{H^2(\widetilde{D})}+h e^{-\a'\tau h}\| u_{reg}\|_{H^2(\widetilde{D})}\nonumber\\
& +h e^{-\a' \tau h}(\| \p_\nu (u-u') \|_{L^\infty(\p S^i_Q)}+\tau\| u-u' \|_{L^\infty(\p S^i_Q)})\nonumber\\
&+h(\| \p_\nu (u-u') \|_{L^\infty(\p S^e)}+\tau\| u-u' \|_{L^\infty(\p S^e)})\nonumber \\
&+h^2\| u-u' \|_{L^\infty(\widetilde{D}^e)}+(\tau^{-1}+he^{-\a'\tau h})\| u\|_{H^1(\widetilde{D})},
\end{align}
where $\a'=\cos(\frac{\pi+b}{4})>0$, and $C$ depends only on the parameters $k, \g, \mathcal{Q}, \eta_m, \eta_M, a_m, a_M$. 
\end{proposition}
%%%%%%%%%%%%%%%%%%%%%%%%%%%%%%%%%%%%%%%%%%%%%%%%%%%%%%
\begin{proof}
Using the decomposition \eqref{eq:decom_sol} and the integral identity \eqref{IntId}, one has
\begin{align}\label{eq:46}
\int_{\Gamma^\pm_\infty}  u_0 \p_\nu u_{sing}\, ds &=\ds\int_{\Gamma^\pm_\infty \setminus \Gamma^\pm}u_0 \p_\nu u_{sing}\, ds - \int_{\Gamma^\pm} u_0 \p_\nu u_{reg}\, ds\nonumber \\
 & + \frac{1}{1-\gamma}\int_{\p S^i_Q \cup \p S^e} (u-u')\p_\nu u_0-u_0\p_\nu(u-u')\, ds \nonumber \\
 & - \frac{k^2}{1-\gamma}\int_{\widetilde{D}^e}(u-u')u_0\, dx - \frac{k^2 q}{\gamma(1-\gamma)}\int_{\widetilde{D}}u_0 u\, dx.
\end{align}
From Theorem \ref{theo_decom_sol}, the singular function $u_{sing}$ is harmonic in $\widetilde{D}$. By further using \eqref{eq:helmholtz_total}, it holds in $\widetilde{D}$ that
\begin{equation}
\Delta u_{reg}+\frac{k^2 q}{\gamma} u=0.
\end{equation}
It follows from Green's formula,
\begin{equation}\label{eq:48}
\int_{\Gamma^\pm} u_0 \p_\nu u_{reg} ds=-\int_{\p S^i_D} u_0 \p_\nu u_{reg} ds+\int_{\widetilde{D}}\n u_0 \n u_{reg} dx - \frac{k^2 q}{\gamma} \int_{\widetilde{D}} u_0 u dx.
\end{equation}
Inserting \eqref{eq:48} into \eqref{eq:46}, we obtain \eqref{IntId2}.\\

We derive next the estimates of each term in the right-hand-side of \eqref{IntId2}. We define the integrals $I_1,\cdots, I_9$ as follows,
\begin{align}
I_1 &=\int_{\Gamma^\pm_\infty \setminus \Gamma^\pm} u_0 \p_\nu u_{sing} ds,  & I_2 &=\int_{\p S^i_D} u_0 \p_\nu u_{reg} ds,\nonumber \\
 I_3 &=\int_{\widetilde{D}} \n u_0 \n u_{reg} dx,  &I_4 &=\int_{\p S^i_Q} (u-u') \p_\nu u_0 ds,\nonumber \\
I_5 &=\int_{\p S^i_Q} u_0 \p_\nu (u-u') ds,  &I_6 &=\int_{\p S^e} (u-u') \p_\nu u_0 ds,\nonumber \\
I_7 &=\int_{\p S^e} u_0 \p_\nu (u-u') ds, & I_8 &=\int_{\widetilde{D}^e} (u-u') u_0 dx,\nonumber \\
I_9 &=\int_{\widetilde{D}}u_0 u dx.
\end{align}
From the construction of the CGO solution $u_0$ and Remark \ref{contour}, we have the following properties, for $x\in \widetilde{D}\cup \p S^i_Q \subset \widetilde{Q}$,
\begin{equation}\label{control_u0}
|u_0(x)|\leq e^{\Re(\rho)\cdot(x-x_c)}\leq e^{-\a'\tau r},
\end{equation}
and for $x\in \widetilde{D}^e \cup \p S^e$,
\begin{equation}\label{control_u_0_Se}
|u_0(x)|\leq e^{\Re(\rho)\cdot(x-x_c)} \leq e^{\tau\cdot \frac{1}{\tau}} = e.
\end{equation}
Using \eqref{control_u0}, \eqref{control_u_0_Se} and the Sobolev embedding $H^1 \hookrightarrow L^\infty$ in $\R^2$, it is shown in the proof of Proposition 4.3 in \cite{moi5},
\begin{align}\label{eq:int_est_1}
|I_1| &\leq C_1 |K| \tau^{-\eta}e^{-\a'\tau h/2}, & |I_2| &\leq C_2 h e^{-\a'\tau h}\| u_{reg}\|_{H^2(\widetilde{D})},\nonumber\\
|I_3| &\leq C_3 (h e^{-\a'\tau h}+\frac{1}{\tau})\| u_{reg}\|_{H^2(\widetilde{D})}, & |I_4| &\leq C_4 \tau h e^{-\a' \tau h} \| u-u'\|_{L^\infty(\p S^i_Q)},\nonumber\\
|I_5| &\leq C_5 h e^{-\a' \tau h} \| \p_\nu(u-u')\|_{L^\infty(\p S^i_Q)}, & |I_6| &\leq C_6 \tau h \| u-u'\|_{L^\infty(\p S^e)},\nonumber\\
|I_7| &\leq C_7 h \| \p_\nu(u-u')\|_{L^\infty(\p S^e)},
\end{align}
where the constants $C_1,\cdots, C_7$ depend only on the {\it a-priori} parameters.

The least two integrals $I_8,I_9$ can be estimated by following the same technique above and we have
\begin{align}\label{eq:int_est_2}
&|I_8| \leq C_8 h^2 \| u-u'\|_{L^\infty(\widetilde{D}^e)},&|I_9|\leq C_9 (he^{-\a'\tau h}+\frac{1}{\tau}) \| u\|_{H^1(\widetilde{D})},
\end{align}
where the constants $C_8,C_9$ depend only on the {\it a-priori} parameters.

Regrouping the estimates in \eqref{eq:int_est_1}, \eqref{eq:int_est_2}, one can readily show that the inequality \eqref{UpperBonds} holds. 

The proof is complete. 
\end{proof}

\begin{proposition}\label{Propo_Lower}
Under the same assumptions in Proposition \ref{Propo_Upper}, it holds that
\begin{equation}\label{LowerBound}
\left|\int_{\Gamma^\pm_\infty}u_0 \p_\nu u_{sing} d\sigma \right|=|K| \Gamma(\eta)\left|\phi'(\theta^+)e^{\mathrm{i} a\eta}-\phi'(\theta^-)\right|\tau^{-\eta}\geq |K| \Gamma(\eta)\sin(a\eta)\tau^{-\eta},
\end{equation}
where $\theta^\pm$ signify the arguments of the vectors along $\Gamma^\pm$.
\end{proposition}
\begin{proof}
See Proposition 4.4 in \cite{moi5}.
\end{proof}
%%%%%%%%%%%%%%%%%%%%%%%%%%%%%%%%%%%%%%%%%%%%%%%%%%%%%%%%%
\subsection{Proof of Theorem \ref{main-theorem}}
\begin{proof}
{We begin the proof by recalling some relevant results in Theorem \ref{theo_decom_sol}. For any vertex $x_0$ of the polygon $D$ or $D'$, the corresponding solutions $u$ or $u'$ in a neighborhood of $x_0$ can be decomposed into the following form,
\begin{equation}\label{eq:decom2}
    u(x)=u_{sing}(x)\zeta(r)+u_{reg}(x)=K r^\eta \phi(\theta) \zeta(r)+u_{reg}(x),
\end{equation}
where $(r,\theta)$ represent the local polar coordinate centered at $x_0$. As the polygons $D$ and $D'$ satisfy the admissibility Definition \ref{def:class}, we can thus set $\varrho=l/5$ in the following proof such that the cut-off function $\zeta$ at every vertex satisfies the point (6) in Theorem \ref{theo_decom_sol}. Moreover, the singularity exponents $\eta$ are explicitly determined by \eqref{eq:eta_gamma}. Hence there exist $0<\eta_m<\eta_M<1$, which depend only on $\gamma,a_m,a_M$ such that $\eta_m \leq \eta \leq \eta_M$.

Let $x_c$ be a vertex of $\p D$ such that $\h=\mathrm{dist}(x,D')$. We define the integral contours in Subsection \ref{LocalA} with a radius $h=\min(\h/2,l/5)$ and we reserve all the notations in Subsection \ref{LocalA}. The fact $h \leq l/5 \leq \varrho$ implies that the corner singularity decomposition \eqref{eq:decom2} holds for the solution $u$ near $x_c$ and $\zeta \equiv 1$ in $\B$. Thus Proposition \ref{Propo_Upper} and Proposition \ref{Propo_Lower} hold as well. }

Next we estimate the following norms appearing in the inequality \eqref{UpperBonds}: $\| u_{reg} \|_{H^2(\widetilde{D})}$, $\| u \|_{H^1(\widetilde{D})}$, $\| u-u' \|_{L^\infty(\p S^i_Q)}$, $\| \p_\nu (u-u') \|_{L^\infty(\p S^i_Q)}$, $\| u-u' \|_{L^\infty(\p S^e)}$, $\| \p_\nu (u-u') \|_{L^\infty(\p S^e)}$ and $\| u-u' \|_{L^\infty(\widetilde{D^e})}$.

The estimates of $\| u_{reg} \|_{H^2(\widetilde{D})}$ and $\| u \|_{H^1(\widetilde{D})}$ can be obtained by a direct application of Theorem \ref{theo_decom_sol} and  the well-posedness of the forward problem,
\begin{equation}\label{estimate_D}
\| u_{reg} \|_{H^2(\widetilde{D})}, \| u \|_{H^1(\widetilde{D})} \leq C \| u \|_{H^1(B_R)} \leq C C_fS :=C_{\widetilde{D}},
\end{equation}
where the constant $C_{\widetilde{D}}$ depends only on the {\it a-priori} parameters.

It also follows directly from the Sobolev embedding $H^1 \hookrightarrow L^\infty$ in $\R^2$ and \eqref{eq:well_forward}-\eqref{eq:well_forward2} that 
\begin{equation}\label{estimate_u-u'_i}
\| u-u'\|_{L^\infty(\p S^i_Q)} \leq C(R) \| u-u' \|_{H^1(B_R)} \leq C(R) C_f S:=C_{\p S^i,0},
\end{equation}
where the constant $C_{\p S^i,0}$ depends only on the {\it a-priori} parameters.

We now estimate $\| \n(u-u') \|_{L^\infty(\p S^i_Q)}$. It follows from Theorem \ref{theo_decom_sol} that for all $x\in \p S^i_Q$, 
\begin{equation}\label{eq:bb1}
    |\n u(x)|\leq |K|\eta|x-x_c|^{\eta-1}+|\n u_{reg}(x)|,
\end{equation}
as well as 
\begin{equation}\label{eq:bb2}
    |\n u'(x)|\leq |K'|\eta'|x-x'_c|^{\eta'-1}\zeta'(x)+|K'||x-x'_c|^{\eta'}|\n \zeta'(x)|+|\n u'_{reg}(x)|,
\end{equation}
where $x'_c$ signifies the nearest vertex to $x_c$ on $\p D'$ and $K',\eta',\zeta',u'_{reg}$ are the corresponding elements of $u'$ in the corner singularity decomposition \eqref{eq:decom2}. We next apply the Sobolev embedding $H^1 \hookrightarrow L^\infty$ in $\R^2$ and follow analogously as \eqref{estimate_D}, \eqref{eq:bb1} to derive that
\begin{align*}
    |\n u(x)| & \leq |K| h^{\eta-1} + C(R)\| \n u_{reg} \|_{H^1(\widetilde{D^e})} \leq |K| h^{\eta-1} + C(R)\| u_{reg} \|_{H^2(\widetilde{D^e})},\\
    &\leq |K| h^{\eta-1} + C C(R)\| u \|_{H^1(B_R)} \leq |K| h^{\eta-1} + C C(R) C_f S.
\end{align*}

The definition of $h$ implies $h\leq \mathfrak{h}/2$, in other words, the distance for any points on $D'$ to the ball $\mathcal{B}$ is at least $h$, and therefore $|x-x'_c| \geq h$. By virtue of the cut-off function $\zeta'$, we only need to consider the case $\varrho \leq |x-x'_c|\leq 2\varrho=\frac{2l}{5}$, which in turn means that we only need to estimate the second term in the right hand side of \eqref{eq:bb2}. It is remarked that the cut-off function $\zeta'$ depends only on $l$. Then by applying the same Sobolev embedding, we have
\begin{align*}
    |\n u'(x)|&\leq |K'| h^{\eta'-1}+ |K'|\frac{2l}{5}\| \n \zeta' \|_{L^\infty(\R^2)} + C(R)\| u'_{reg} \|_{H^2(\widetilde{D^e})},\\
    &\leq  |K'| h^{\eta'-1}+ |K'|C(l) + C C(R) C_f S. 
\end{align*}
Summing up the above estimates, we can deduce the following estimate:
\begin{equation}\label{ddd4}
\| \n(u-u')\|_{L^\infty(\p S^i_Q)} \leq |K| h^{\eta-1}+|K'|h^{\eta'-1}+|K'|C(l)+ 2CC(R)C_fS,
\end{equation}
Furthermore, the singularity coefficient $K'$ can be estimated using the singular decomposition \eqref{eq:decom2}. By a direct integration of the singular function over $B_R$, there exists a constant $C(l,\eta_m,\eta_M)>0$ such that,
\[C(l,\eta_m,\eta_M)|K'| \leq \|u'_{sing}\zeta'\|_{L^2(B_R)}\leq \|u'\|_{H^1(B_R)}\leq C_f S.\]
Then it follows that $|K'|\leq C_K S$ and hence \eqref{ddd4} becomes,
\begin{equation}\label{estimate_gra_u-u'_i}
\| \n(u-u')\|_{L^\infty(\p S^i_Q)} \leq |K| h^{\eta-1}+C_K S h^{\eta'-1}+ C_{\p S^i,1}S,
\end{equation}
where the constants $C_K, C_{\p S^i,1}$ depend only on the {\it a-priori} parameters.

By Theorem \ref{theo_decom_sol} we see that $u$ and $u'$ are at least of class $\mathcal{C}^{\eta_m}$ in the neighborhood of $x_c$. The estimations of the $\mathcal{C}^{\eta_m}$ norms of $u$ and $u'$ can be deduced from the decomposition formulas and the Sobolev embedding $H^2(\widetilde{D^e}) \hookrightarrow \mathcal{C^\a}(\widetilde{D^e})$ for all $0<\a <1$ (see Corollary 7.11 in \cite{Gilbarg}). Using the same arguments as the estimations above, we have then,
\begin{equation}\label{eq:ddd3}
    \|u-u'\|_{\mathcal{C}^{\eta_m}(\widetilde{D^e})}\leq |K|+C_T S,
\end{equation}
where $C_T$ depends only on the {\it a-priori} data.
Nevertheless, we can apply again Theorem \ref{theo_decom_sol} and the higher regularity result in \cite{nicaise1990polygonal} to derive that the functions $x \mapsto |x-x_c|\n u(x)$ and $x \mapsto |x-x_c|\n u'(x)$ are in fact at least of class $\mathcal{C}^{\eta_m}$ and their $\mathcal{C}^{\eta_m}$ norms are also bounded by the right-hand side of \eqref{eq:ddd3}. We apply now Proposition \ref{propagation3} with $\a_0=\a_1=\eta_m$. Then there exists $\e_m>0$ depending only on the {\it a-priori} parameters such that if $\e<\e_m$, for all $x\in \widetilde{D}^e \cup \p S^e$,
\begin{equation}\label{estimate_u-u'_e}
|u-u'|(x)\leq \widetilde{C_0} (|K| + C_T S) \left(\ln\ln \frac{S}{\e}\right)^{-\eta_m}.
\end{equation}
In what follows, we denote by $\d(\e)$ the quantity $(\ln\ln \frac{S}{\e})^{-\eta_m}$. It also follows from Proposition \ref{propagation3} and Remark \ref{contour} that for all $x\in \p S^e$, $\frac{1}{2\tau} \leq  \mathrm{dist}(x,\p Q)$ and thus 
\begin{equation}\label{estimate_gra_u-u'_e}
|\n(u-u')|(x)\leq 2 \widetilde{C_1} (|K|+C_T S) \tau\d(\e).
\end{equation}

Next we apply Propositions \ref{Propo_Upper} and \ref{Propo_Lower} with the estimates \eqref{estimate_D}, \eqref{estimate_u-u'_i}, \eqref{estimate_gra_u-u'_i}, \eqref{estimate_u-u'_e} and \eqref{estimate_gra_u-u'_e}. We absorb into the left hand side all the constants depending only on the {\it a-priori} parameters. There exists a constant $C$ depending only on the {\it a-priori} parameters such that
\begin{align*}
    C|K|\tau^{-\eta} \leq & |K|\tau^{-\eta}e^{-\a'\tau h/2}+S\tau^{-1}+Sh e^{-\a' \tau h} + (|K| h^{\eta-1}+S h^{\eta'-1}+S+S\tau )h e^{-\a' \tau h}\\
    & +(|K|+S)h\tau \d(\e)+(|K|+S)h^2\d(\e).
\end{align*}
%\[ C\tau^{-\eta}\leq \tau^{-\eta} e^{-\a'\tau h/2}+\tau^{-1}+h e^{-\a' \tau h}+ h e^{-\a' \tau h}(h^{\eta_m-1}+1+\tau)+h\tau \d(\e)+h^2\d(\e).\]
Using the facts that $h\leq 1$ and $\tau \geq 1$ and the inequalities $e^{-x}\leq x^{-1}$, $e^{-x}\leq x^{-2}$ for all $x>0$, we multiply on the two sides the factor $\frac{\tau^\eta}{|K|}$ which yields that
\begin{align}\label{eq:ddd1}
    C& \leq e^{-\a'\tau h/2}+ \frac{S}{|K|}\tau^{\eta-1}+ \frac{S}{|K|}h\tau^\eta e^{-\a' \tau h}+(h^\eta+\frac{S}{|K|}h^{\eta'}+\frac{S}{|K|}h+\frac{S}{|K|}h\tau)\tau^\eta e^{-\a' \tau h} \nonumber\\
    & +(1+\frac{S}{|K|})h\tau^{\eta+1}\d(\e)+(1+\frac{S}{|K|})h^2\tau^\eta\d(\e),\nonumber\\
    &\leq h^{-1}\tau^{-1}+\frac{S}{|K|}\tau^{\eta-1}+\frac{S}{|K|}h^{\eta-1}\tau^{\eta-1}+\frac{S}{|K|}h^{\eta'-1}\tau^{\eta-1}+\frac{S}{|K|}\tau^{\eta-1}+\frac{S}{|K|}h^{-1}\tau^{\eta-1}\nonumber\\
    & +(1+\frac{S}{|K|})h\tau^{\eta+1}\d(\e),\nonumber\\
    &\leq (1+\frac{S}{|K|})(h^{-1}\tau^{\eta-1}+h\tau^{\eta+1}\d(\e)).
\end{align}

%\begin{equation}\label{eq:ddd1}
%\begin{split}
%C & \leq e^{-\a'\tau h/2}+ \tau^{\eta-1}+ h\tau^\eta e^{-\a' \tau h}+(h^{\eta_m}\tau^\eta+h\tau^\eta+h\tau^{1+\eta})e^{-\a' \tau h}+h\tau^{\eta+1}\d(\e)+h^2\tau^\eta \d(\e), \\
%& \leq h^{-1}\tau^{-1}+\tau^{\eta-1}+ \tau^{\eta-1}+h^{\eta_m-1}\tau^{\eta-1}+\tau^{\eta-1}+h^{-1}\tau^{\eta-1}+h\tau^{\eta+1}\d(\e),\\
%& \leq h^{-1}\tau^{\eta-1}+h\tau^{\eta+1}\d(\e).
%\end{split}
%\end{equation}
We next determine a minimum modulo constants of the right hand side of the inequality in \eqref{eq:ddd1}. Set $\tau=\tau_e$ with
\begin{equation}\label{eq:tau_e}
\tau_e=h^{-1}\d(\e)^{-1/2}.
\end{equation}
It is straightforward to verify that for $\e$ smaller than a certain constant one has that if
\begin{equation}\label{check_tau}
\d(\e)^{-1/2}\geq C_{\tau_0},
\end{equation}
where $C_{\tau_0}$ is defined in Remark \ref{contour}, then 
\[\tau_e \geq \tau_0,\]
which justifies that we can take $\tau=\tau_e$ in \eqref{eq:ddd1}.

Solving for $h$, it gives
\begin{equation}\label{eq:ddd2}
h\leq C(1+\frac{S}{|K|})^{\frac{1}{\eta}} \left(\ln\ln \frac{S}{\e}\right)^{\frac{\eta_m(\eta-1)}{2\eta}},
\end{equation}
and thus,
\begin{equation}\label{final}
\min(\h/2,l/5) \leq C(1+\frac{S}{K_m})^{\frac{1}{\eta_m}} \left(\ln \ln \frac{S}{\e}\right)^{\frac{\eta_m(\eta_M-1)}{2\eta_M}}.
\end{equation}
Hence, if $\e$ is small enough such that $\e < \e_m$ in Proposition \ref{propagation3}, that (\ref{check_tau}) holds and that the right hand side of (\ref{final}) is smaller than $l/5$, we have
\[\h \leq C(1+\frac{S}{K_m})^{\frac{1}{\eta_m}} \left( \ln\ln \frac{S}{\e} \right)^{\frac{\eta_m(\eta_M-1)}{2\eta_M}}.\]
Therefore, the claim of this theorem readily follows.

The proof is complete. 
\end{proof}

\section{Implications to invisibility and transmission eigenvalue problem}\label{sec:scatter_stable}

In this section, we present two interesting implications to the wave scattering theory which are byproducts of our stability study. Let us consider the scattering problem \eqref{eq:helmholtz_total} and \eqref{eq:asympt_far}. It is said that invisibility occurs if 
\begin{equation}\label{eq:i1}
u_\infty\equiv 0. 
\end{equation}
The first byproduct we shall establish is to show that if a general scattering medium possesses a corner on its support, then the $L^2$-energy of the corresponding far-field pattern possesses a positive lower bound. That is, a corner scatters stably and invisibility cannot be achieved. In fact, we have
\begin{theorem}\label{th:corner_scatter}
Let $D$ be a Lipschitz domain in $\R^2$, not necessarily a convex polygon. We assume that $\p D$ admits a convex polygonal point, i.e. there exists $x_c \in \p D$ such that $D \cap B_l(x_c)$ is a plan sector for $l>0$. We denote by $0<a<\pi$ the opening of the corner $x_c$. Let $\gamma>0$, $\gamma \neq 1$ and $q\in L^\infty(\R^2)$. We consider the scattering problem \eqref{eq:helmholtz_total} where $\sigma$ is defined by \eqref{eq:sigma_form} and the incident field $u^i\in H^2_{loc}(\R^2)$ is a nontrivial entire solution to \eqref{eq:helmholtz_inc} in $\R^2$. We suppose the following conditions are fulfilled:
\begin{enumerate}
    \item $D \Subset B_R$ for some $R>l$;
    \item $0< a_m \leq a \leq a_M < \pi$;
    \item $0<\gamma_m \leq \g \leq \gamma_M$;
    \item $\mathrm{supp}(q-1) \Subset B_R$;
    \item $\mathrm{supp}(q-1) \cap B_l(x_c) \subset D \cap B_l(x_c)$;
    \item $\| q \|_{L^\infty(\R^2)}\leq \mathcal{Q}$;
    \item $\| u^i \|_{H^2(B_{2R})}\leq S$ for some $S>0$;
    \item The solution $u$ to \eqref{eq:helmholtz_total} admits a non-degenerate corner singularity near $x_c$, i.e. the singularity coefficient $K \neq 0$ in Theorem \ref{theo_decom_sol}. And the singularity exponent is denoted by $0<\eta<1$.
\end{enumerate}
Then there exists a constant $C$ depending only on the {\it a-priori} data $k$, $R$, $\gamma_m$, $\gamma_M$, $a_m$, $a_M$, $l$ and $\mathcal{Q}$ such that one has
\begin{equation}\label{corner_scatter}
\| u_\infty \|_{L^2(\S^1)} \geq \frac{S}{\exp \exp \left( C(1+\frac{S}{|K|})^{\frac{2}{\eta(1-\eta)}} \right)}.
\end{equation}
\end{theorem}

\begin{proof}
The proof follows from similar arguments in the proof of Theorem \ref{main-theorem} with necessary modifications. We take $D'=\emptyset$, $q' \equiv 1$ in $\R^2$, and therefore, $\e=\| u_\infty \|_{L^2(\S^2)}$. In the current scenario, we only need to consider the effect caused by the vertex $x_c$, and the corresponding reasoning in the previous sections could be therefore considerably reduced. We set $h=l/5$ and $\d(\e)=(\ln\ln \frac{S}{\e})^{-\eta}$. The inequality \eqref{eq:ddd1} still holds, and the minimum modulo constants of its right hand side occurs at $\tau_e$ given by \eqref{eq:tau_e}. We define the constant $\e'_m$ by \eqref{check_tau}, namely
\begin{equation}\label{critere2}
    \d(\e'_m)^{-1/2}=C_{\tau_0}.
\end{equation}
Then we set the constant $\e_{min}:=\min(\e_m,\e'_m)$, where $\e_m$ is given in Proposition \ref{propagation3}. Furthermore, it follows from \eqref{critere2} and \eqref{eq:passage6} that $\e_{min}$ can be expressed into the following way,
\begin{equation}\label{frac_emin}
    \e_{min}=\frac{S}{\exp \exp C_{min}},
\end{equation}
where $C_{min}$ depends only on {\it a-priori} data.

In the case $\e \leq \e_{min}$, we can take $\tau=\tau_e$ in \eqref{eq:ddd1}. After solving for $h$, we can obtain a similar result as \eqref{eq:ddd2},
\begin{equation}
    h \leq C(1+\frac{S}{|K|})^{\frac{1}{\eta}} (\ln\ln \frac{S}{\e})^{\frac{\eta-1}{2}},
\end{equation}
where the constant $C$ depends only on the {\it a-priori} data.

Thus, one has
\begin{equation}\label{eq:ddd5}
    \| u_\infty \|_{L^2(\S^1)}=\e \geq \frac{S}{\exp \exp \left( C(1+\frac{S}{|K|})^{\frac{2}{\eta(1-\eta)}} \right)}.
\end{equation}

Combining with \eqref{frac_emin} and updating the constant $C$ in \eqref{eq:ddd5}, the inequality \eqref{corner_scatter} holds readily.

The proof is complete. 
\end{proof}
%\begin{remark}
%Theorem \ref{th:corner_scatter} shows that the presence of a polygonal corner assure the scattering phenomena under the genetic admissible assumptions. Moreover, we give in \eqref{corner_scatter} the dependence of the lower bound of the far-field pattern $u_\infty$ with the geometrical characteristics $l$ and $\eta$.
%\end{remark}
%\begin{remark}
%In the context of the {\it transmission eigenvalue problems}, they consist to study the solutions to \eqref{eq:helmholtz_total} with zero far-field pattern. Cakoni and Xiao derived a uniqueness result on the transmission eigenvalue problem in \cite{cakoni2019}. Compare to our stability result, the uniqueness is an immediate consequence of Theorem \ref{th:corner_scatter}.
%\end{remark}

Next, we consider the implication to the transmission eigenvalue problem. For the scattering problem \eqref{eq:helmholtz_total}, if invisibility occurs, namely \eqref{eq:i1} holds, one can directly verify by setting $v=u^i|_D$ that 
\begin{equation}\label{eq:trans1}
\begin{cases}
 \div (\sigma\n u)+k^2 q u=0 & \text{in  } D,\medskip\\
 \Delta v+k^2 v=0 & \text{in  } D,\medskip\\
 u=v,\quad\sigma\p_\nu u=\p_\nu v & \text{on  }\p D. 
\end{cases}
\end{equation}
On the other hand, if there exists nontrivial solutions $u\in H^1(D)$ and $v\in H^1(D)$ to \eqref{eq:trans1}, then $k\in\R_+$ is called a transmission eigenvalue and $u, v$ are the associated transmission eigenfunctions (cf. \cite{cakoni2016}). Let $v$ be transmission eigenfunction associated with $k\in\R_+$. Then the following so-called Herglotz approximation holds, which represents a certain Fourier extension property (cf. \cite{DDL,DCL}). For any $\e \ll 1$, there exists $g_\e \in L^2(\S^1)$ such that (cf. \cite{weck2004})
\begin{equation}\label{eq:app1}
\|v_{g_\e}-v\|_{H^1(D)}\leq \e, \ \ v_{g_\e}(x):=\int_{\S^1} e^{\mathrm{i}kx\cdot d} g_\e (d)\, ds(d), 
\end{equation}
where $v_{g_\e}$ is referred to as a Herglotz wave function associated with the kernel density $g_\e$ and is an entire solution to the Helmholtz equation \eqref{eq:helmholtz_inc}. 

\begin{theorem}\label{thm:main2}
Let $(D,\g, q)$ be a convex polygonal scatter satisfying the assumptions in Theorem \ref{main-theorem}. We consider the transmission eigenvalue problem \eqref{eq:trans1} with nontrivial eigenfunctions $u, v\in H^1(D)$ with $\| v \|_{H^1(D)}=1$. Then it holds that:
\begin{enumerate}
\item If $v$ can be extended outside $D$ to be an entire solution to \eqref{eq:helmholtz_inc}, then $u\in H^2(D)$.
\item Let $(v_{g_\e})_{\e>0}$ be a family of Herglotz wave functions that are given in \eqref{eq:app1} and approximate the transmission eigenfunction $v$. There exist positive constants $A$ and $b$ which are independent of $g_\e$ such that for any given function $\psi: \R_+\rightarrow \R_+$ satisfying $\psi(\e)\rightarrow 0$ as $\e\rightarrow 0$, if $\|g_\e\|_{L^2(\S^1)}\leq \psi(\e)\left( \ln\ln\frac{A}{\e} \right)^b$, then it holds that
\begin{equation}\label{vanish}
\lim_{x\neq x'\in D ,x,x'\rightarrow x_c}\frac{|u(x)-u(x')|}{|x-x'|^\eta}=0
\end{equation}
at each vertex $x_c$ of $D$, where $\eta$ is given by \eqref{eq:eta_gamma}.
\end{enumerate}
\end{theorem}

\begin{remark}
The Herglotz approximation property \eqref{eq:app1} can be regarded as a certain Fourier extension of the transmission eigenfunction $v$; see \cite{DDL,DCL} for more related discussion. On the other hand, if $v$ can be extended to an entire solution to \eqref{eq:helmholtz_inc}, it is clear that $v$ is analytic in $\mathbb{R}^2$. Hence, intriguingly, Theorem~\ref{thm:main2} connects the local regularity of the transmission eigenfunction $u$ around a corner to the analytic or Fourier extension property of the transmission eigenfunction $v$. 

Note that by the standard Sobolev embedding, $u\in H^2(D)$ implies that $u$ is H\"older continuous up to the boundary. Hence, the first result in the theorem indicates that if $u$ is not H\"older continuous up to a corner vertex, then $v$ cannot be analytically extended across the corner. Moreover, it further indicates that if $v$ is approximated by the Herglotz sequence $v_{g_\e}$, then $\|g_\e\|_{L^2(\mathbb{S}^1)}$ must be unbounded, since otherwise $v_{g_\e}$ converges, by passing to a subsequence if necessary, to an entire solution to \eqref{eq:helmholtz_inc}, say $v_0$, which is an analytic extension of $v$ into $\mathbb{R}^2$; see also \cite{BL1,cakoni2019} for more relevant discussion about this point.  

The second result in Theorem~\ref{thm:main2} quantitatively reinforces the above assertion about the unboundedness of the Herglotz kernels $\|g_\e\|_{L^2(\mathbb{S}^1)}$ if the transmission eigenfunction $u$ is not H\"older continuous up to any corner vertex of the domain $D$. To illustrate this point, let us take $\psi(\varepsilon)=\left( \ln\ln\frac{A}{\e} \right)^{-b/2}$ in item (2) in the theorem. Then if $u$ is not $C^\eta$ up to a corner vertex, then one clearly has $\|g_\e\|_{L^2(\S^1)}\geq \psi(\e)\left( \ln\ln\frac{A}{\e} \right)^{b/2}\rightarrow +\infty$ as $\e\rightarrow +0$. 

It is also interesting to point out that Theorem~\ref{thm:main2} corroborates the studies in \cite{DDL,DCL}, which make use of a similar Herglotz approximation property as a regularity criterion in a certain different setup from the current article. 

Finally, we would like to point out that Theorem~\ref{thm:main2} still holds true for the case that $D$ is a generic domain but possess an admissible corner. In fact, as point out earlier, our argument can be localised around the corner and is irrelevant to the rest part of $\partial D$. Nevertheless, in order to ease the exposition, we only consider the case that $D$ is a convex polygon. 

\end{remark}

\begin{proof}[Proof of Theorem~\ref{thm:main2}]
{
We assume that $v$ can be extended to be an entire solution to the Helmholtz equation \eqref{eq:helmholtz_inc}, which is still denoted by $v$. It follows directly from the classic elliptic regularity that $v$ is real analytic on all compact domains in $\R^2$. 

Consider the scattering problem \eqref{eq:helmholtz_total} with $u^i=v$. We denote also by $u$ the extension by $v$ of $u$, the transmission conditions \eqref{eq:trans1} on $\p D$ implies that $u$ satisfies \eqref{eq:helmholtz_total} in $\R^2$ and  
\begin{equation}\label{zero}
    u_\infty(\hat x; v)\equiv 0.
\end{equation}
We then apply Theorem \ref{th:corner_scatter} at each vertex of the polygon $D$. The condition \eqref{zero} holds only if $K=0$ at each vertex. From the singular decomposition Theorem \ref{theo_decom_sol}, we have $u=u_{reg}\in H^2(D)$. Hence, the claim (1) now follows. 

We proceed to prove the claim (2). Let $v^i_0$ be the zero-extension of $v$ in $\R^2$, and let $u^s_0$ be the radiating solution to $\div(\sigma \n u^s_0)+k^2 q u^s_0=-\div((\g-1)\chi_D \n v^i_0)-k^2(q-1)v^i_0$. Denoting by $w$ the zero-extension of $u-v\in H^1_0(D)$ in $\R^2$, it follows from the standard scattering theory, $u^s_0=w$ because
\begin{align*}
   & \div(\sigma \n u^s_0)+k^2 q u^s_0 = -\div((\g-1)\chi_D \n v^i_0)-k^2(q-1)v^i_0 \\
    &=-\div((\g-1)\chi_D \n v)-k^2(q-1)v =\div(\sigma \n w)+k^2 q w
\end{align*}
in $\R^2$ and via the Sommerfeld radiation condition. As an immediate consequence, the far-field pattern of $u^s_0$ is zero.

Let $g_\e \in L^2(\mathbb{S}^1)$ and $v_{g_\e}$ be the Herglotz wave given by \eqref{eq:app1}. We consider the scattering problem from the scatterer $(D;\g,q)$ associated with the incident wave $v_{g_\e}$.  We denote by $u_\e$ the solution to \eqref{eq:helmholtz_total} and by $u^s_\e$ the scattered wave in this scenario. Since $v_{g_\e}$ approximates $v$ in $H^1(D)$ and that $q$ is supported in $D$, one clearly has that $-\div((\g-1)\chi_D \n v_{g_\e})-k^2(q-1)v_{g_\e}$ approximates $-\div((\g-1)\chi_D \n v^i_0)-k^2(q-1)v^i_0$ in $\R^2$. Then $u_\e$ approximates $u$ in $H^1(D)$ and the far-field pattern of $u^s_\e$ approximates that of $u^s_0$. We apply again the standard scattering theory, there exists a constant $C_{(D;\g,q),k}>0$ such that,
\begin{equation}\label{eq:ar1}
\|u_\infty(\hat x; v_{g_\e})\|_{L^2(\S^1)}\leq C_{(D;\g,q),k}\|v-v_{g_\e}\|_{H^1(D)}\leq C_{(D;\g,q),k} \e. 
\end{equation}

On the other hand, letting $x_c\in \p D$ be a vertex of $D$, we apply Theorem~\ref{th:corner_scatter} on $x_c$. It follows from \eqref{corner_scatter},
\begin{equation}\label{eq:ar2}
\|u_\infty(\hat{x},v_{g_\e})\|_{L^2(\mathbb{S}^1)} \geq \frac{\mathbf{S}_{R,k}\| g_\e \|_{L^2(\S^1)}}{\exp \exp \left( C(1+\frac{\mathbf{S}_{R,k}\| g_\e \|_{L^2(\S^1)}}{|K(x_c,v_{g_\e})|})^{\frac{2}{\eta(1-\eta)}} \right) },
\end{equation}
where we remark that the amplitude $S$ of an Herglotz incident wave satisfies $S=\mathbf{S}_{R,k} \|g_\e\|_{L^2(\S^1)}$. We estimate here $\|g_\e\|_{L^2(\S^1)}$ for the first time. Using \eqref{eq:app1} and the fact $\|v\|_{H^1(D)}=1$, we deduce that $\|g_\e\| \geq c_D>0$ for $\e$ small enough. Combining \eqref{eq:ar1}, \eqref{eq:ar2} and the above estimation, we can obtain, for $\e$ small enough,
\begin{equation}\label{eq:ar3}
\frac{\widetilde{S}}{\e}:=\frac{c_D \mathbf{S}_{R,k}}{C_{(D;,\g,q),k}\e}\leq \exp \exp \left( C(1+\frac{\mathbf{S}_{R,k}\| g_\e \|_{L^2(\S^1)}}{|K(x_c,v_{g_\e})|})^{\frac{2}{\eta(1-\eta)}} \right).
\end{equation}
Then solving for $|K(x_c,v_{g_\e})|$ and using the property $0<\eta_m\leq \eta\leq \eta_M<1$, \eqref{eq:ar3} implies, for $\e$ small enough again,
\begin{align}\label{eq:ar4}
|K(x_c,v_{g_\e})| & \leq 2 C^{\frac{\eta(1-\eta)}{2}} \mathbf{S}_{R,k}\|g_\e\|_{L^2(\S^1)} \left( \ln\ln \frac{\widetilde{S}}{\e} \right)^\frac{\eta(\eta-1)}{2} \nonumber\\
&\leq 2C^{1/8}\mathbf{S}_{R,k}\|g_\e\|_{L^2(\S^1)} \left( \ln\ln \frac{\widetilde{S}}{\e} \right)^\frac{\eta_m(\eta_M-1)}{2}.
\end{align}
Setting now $A=\widetilde{S}$ and $b=\frac{\eta_m(1-\eta_M)}{2}$, we have,
\begin{equation}\label{eq:ar5}
|K(x_c,v_{g_\e})| \leq 2C^{1/8}\mathbf{S}_{R,k} \psi(\e)\underset{\e\rightarrow 0}{\longrightarrow} 0.
\end{equation}

Next, it follows from Theorem~\ref{theo_decom_sol} that the solution $u_\e$ admits the singular decomposition \eqref{eq:decom2} at each vertex $x_c$. Then, it holds for $x \neq x' \in D$,
\begin{equation}\label{eq:ar6}
\frac{|u_\e(x)-u_\e(x')|}{|x-x'|^\eta} \leq C_\eta |K(x_c,v_{g_\e})|+\frac{|u_{\e,reg}(x)-u_{\e,reg}(x')|}{|x-x'|^\eta},
\end{equation}
where $C_\eta=\| r^\eta \phi(\theta) \|_{\mathcal{C}^\eta(D)}$ depending only on $\eta$.

Using \eqref{eq:app1}, the estimation \eqref{eq:estime_u_reg} to $u_{\e,reg}$, the standard scattering theory and the Sobolev embedding $H^2(D)\hookrightarrow \mathcal{C}^\a(D)$ for all $0<\a<1$ we have for $x \neq x' \in D$,
\begin{align}\label{eq:ar7}
|u_{\e,reg}(x)-u_{\e,reg}(x')| &\leq \|u_{\e,reg}\|_{\mathcal{C}^\a(D)}|x-x'|^\a  \leq \widetilde{C}\|u_{\e,reg}\|_{H^2(D)}|x-x'|^\a \nonumber\\
&\leq \widetilde{C}\|u_\e\|_{H^1(B_R)}|x-x'|^\a \leq \widetilde{C}\|g_\e\|_{L^2(\S^1)}|x-x'|^\a,
\end{align}
where the constant $\widetilde{C}$ depends only on the {\it a-priori} data. Combining \eqref{eq:ar5}, \eqref{eq:ar6}, \eqref{eq:ar7} and setting $\a>\eta_M$ and $\e=|x-x'|$ with $|x-x_c|+|x'-x_c|$ small enough, one thus has
\begin{align}\label{eq:ar8}
&\frac{|u_\e(x)-u_\e(x')|}{|x-x'|^\eta} \leq C_\eta |K(x_c,v_{g_\e})|+\widetilde{C}\|g_\e\|_{L^2(\S^1)}|x-x'|^{\a-\eta}\nonumber \\
& \leq 2 C_\eta C^{1/8}\mathbf{S}_{R,k} \psi(|x-x'|)+\widetilde{C}\psi(|x-x'|)\left(\ln\ln\frac{\widetilde{S}}{|x-x'|}\right)^{\frac{\eta_m(1-\eta_M)}{2}}|x-x'|^{\a-\eta}\nonumber\\
&\underset{x \neq x'\rightarrow x_c}{\longrightarrow} 0 .
\end{align}

Next, since $u_\e$ approximates $u$ in $H^1(D)$, it follows from the Sobolev embedding $H^1\hookrightarrow L^\infty$ in $\R^2$ that
\begin{equation}\label{eq:ar9}
\|u-u_e\|_{L^\infty(D)}\leq \hat{C}\|u-u_e\|_{H^1(D)}\leq \hat{C}\|v-v_{g_\e}\|_{H^1(D)}\leq \hat{C}\e.
\end{equation}
Setting again $\e=|x-x'|$ for $|x-x_c|+|x'-x_c|$ small enough and combining \eqref{eq:ar8} and \eqref{eq:ar9}, we can deduce that
\begin{align}
\frac{|u(x)-u(x')|}{|x-x'|^\eta} &\leq \frac{|u(x)-u_\e(x)|}{|x-x'|^\eta}+\frac{|u(x')-u_\e(x')|}{|x-x'|^\eta}+\frac{|u_\e(x)-u_\e(x')|}{|x-x'|^\eta}\nonumber \\
&\leq \hat{C}|x-x'|^{1-\eta}+\hat{C}|x-x'|^{1-\eta}+\frac{|u_\e(x)-u_\e(x')|}{|x-x'|^\eta}\nonumber\\
&\underset{x \neq x' \rightarrow x_c}{\longrightarrow}0, 
\end{align} 
Hence, the claim \eqref{vanish} follows and the proof is complete.
}

\end{proof}

\section*{Acknowledgment}
The work of H. Liu is supported by the Hong Kong RGC General Research Funds (projects 12302919, 12301218 and 11300821), and the France-Hong Kong ANR/RGC Joint Research Grant, A-HKBU203/19.

\end{document}